\documentclass[10pt,reqno]{amsart}
\usepackage{amssymb,amsmath,amscd,amsthm}
\usepackage[top=3.6cm, bottom=3.4cm, left=3.4cm, right=3.4cm]{geometry}
\usepackage{hyperref}

  \makeatletter
    
    \@addtoreset{equation}{section}
  \makeatother

\theoremstyle{plain}
\newtheorem{theorem}{Theorem}[section]                    
\newtheorem{proposition}[theorem]{Proposition}            
                
\newtheorem{lemma}[theorem]{Lemma}

\theoremstyle{definition}
\newtheorem{remark}[theorem]{Remark}

\def\us#1_#2{\underset{#2}{#1}}
\def\os#1^#2{\overset{#2}{#1}}
\def\i<#1>{\langle {#1} \rangle}

\DeclareMathOperator{\spec}{sp}

\newcommand{\DT}{\Delta\bfT}
\newcommand{\act}{\curvearrowright}
\DeclareMathOperator{\lspan}{span}
\DeclareMathOperator{\ospan}{\overline{\lspan}}

\DeclareMathOperator{\id}{id}

\makeatletter 
\@tempcnta\z@
\loop\ifnum\@tempcnta<26
\advance\@tempcnta\@ne
\expandafter\edef\csname frak\@Alph\@tempcnta\endcsname{\noexpand\mathfrak{\@Alph\@tempcnta}}
\expandafter\edef\csname l\@Alph\@tempcnta\endcsname{\noexpand\mathbb{\@Alph\@tempcnta}}
\expandafter\edef\csname cal\@Alph\@tempcnta\endcsname{\noexpand\mathcal{\@Alph\@tempcnta}}
\expandafter\edef\csname rm\@Alph\@tempcnta\endcsname{\noexpand\mathrm{\@Alph\@tempcnta}}
\expandafter\edef\csname bf\@Alph\@tempcnta\endcsname{\noexpand\mathbf{\@Alph\@tempcnta}}
\repeat
\makeatother

\allowdisplaybreaks[4]
\begin{document}
\title[Boundary rigidity for free products]{Boundary rigidity for free product $\rmC^*$-algebras}
\author[K.~Hasegawa]{Kei~Hasegawa}
\address{Graduate~School~of~Mathematics, Kyushu~University, Fukuoka~819-0395, Japan}
\email{k-hasegawa@math.kyushu-u.ac.jp}
\begin{abstract}
For any reduced free product $\rmC^*$-algebra $(A, \varphi) =(A_1, \varphi_1) \star (A_2, \varphi_2)$,
we prove a boundary rigidity result for the embedding of $A$ into its associated $\rmC^*$-algebra $\DT(A, \varphi)$.
This provides new examples of rigid embeddings of exact $\rmC^*$-algebras into purely infinite simple nuclear $\rmC^*$-algebras.
\end{abstract}
\maketitle

\section{Introduction}

By a deep theorem of Kirchberg--Phillips \cite{Kirchberg-Phillips},
any separable exact $\rmC^*$-algebra is embedded into the Cuntz algebra $\calO_2$.
In the present paper, we are interested in rigid or canonical embeddings of exact $\rmC^*$-algebras
into nuclear ones.
Here, an embedding $A \subset B$ of $\rmC^*$-algebras is called \emph{rigid} if
the identity map on $B$ is the unique completely positive map from $B$ to itself which is identical on $A$.
Such rigid embeddings naturally arise from boundary actions;
For any action $\Gamma \act X$ of a discrete group $\Gamma$ on a $\Gamma$-boundary in the sense of Furstenberg \cite{Furstenberg}
the embedding of the reduced group $\rmC^*_{\rm r}(\Gamma)$ into the reduced crossed product $C(X)\rtimes_{\rm r} \Gamma$
is rigid.

For the action $\lF_2 \act \partial \lF_2$ of the free group on the Gromov boundary,
Ozawa \cite{Ozawa-BLMS} proved a stronger rigidity result that the crossed product $C(\partial \lF_2) \rtimes \lF_2$ naturally sits between $\rmC^*_{\rm r}(\lF_2)$ and its injective envelope
$I(\rmC^*_{\rm r}(\lF_2))$.
The injective envelope $I(A)$ of a $\rmC^*$-algebra $A$
is an injective $\rmC^*$-algebra introduced by Hamana \cite{Hamana}
which satisfies that $A \subset I(A)$ is rigid.
In the same paper,
Ozawa conjectured that for every separable exact $\rmC^*$-algebra $A$ there exists a nuclear $\rmC^*$-algebra $N(A)$ such that $A \subset N(A) \subset I(A)$.
For general exact discrete group $\Gamma$,
Kalantar--Kennedy \cite{Kalantar-Kennedy} proved that
one can take $N(\rmC^*_{\rm r}(\Gamma))$ as the crossed product $C(\partial_{\rm F} \Gamma) \rtimes \Gamma$,
where $\partial_{\rm F} \Gamma$ is the Furstenberg boundary of $\Gamma$.
In a recent breakthrough on $\rmC^*$-simplicity,
it turned out that boundary actions and their rigidity play an important role for the analysis of reduced group $\rmC^*$-algebras (see \cite{Kalantar-Kennedy,BKKO}). 
On the other hand, for general $\rmC^*$-algebras beyond the class of group $\rmC^*$-algebras,
there has been no known results so far for Ozawa's conjecture and even for rigid embeddings into nuclear $\rmC^*$-algebras.

In the present paper, we investigate rigid embeddings of reduced free product $\rmC^*$-algebras.
In \cite{Hasegawa-BS} we introduced a $\rmC^*$-algebra $\DT(A, \varphi)$ associated with any reduced free product
$(A, \varphi ) =(A_1, \varphi_1) \star (A_2, \varphi_2)$.
For the reduced group $\rmC^*$-algebra $\rmC^*_{\mathrm{r}}(\Gamma_1 \ast \Gamma_2)$ of any free product group $\Gamma_1 \ast \Gamma_2$
with the canonical tracial state $\tau$,
the associated $\rmC^*$-algebra $\DT(\rmC^*_{\mathrm{r}}(\Gamma_1 \ast \Gamma_2), \tau)$
is naturally identified with $C(\DT)\rtimes_{\mathrm{r}}\Gamma_1 \ast \Gamma_2$
associated to a natural action on
the compactification $\DT$ of the Bass--Serre tree associated with $\Gamma_1 \ast \Gamma_2$ (see \cite{Serre} and \cite{Bowditch} for the
Bass--Serre theory and compactifications of trees, respectively).
When $\Gamma_1$ and $\Gamma_2$ are infinite, then
$\DT$ is a $\Gamma_1 \ast \Gamma_2$-boundary,
which implies the embedding $\rmC^*_{\rm r}(\Gamma_1 \ast \Gamma_2) \subset \DT(\rmC^*_{\rm r}(\Gamma_1 \ast \Gamma_2) , \tau)$ is rigid.
We prove the following generalization for reduced free product $\rmC^*$-algebras with respect to GNS-essential states.
Here, we say that a state on a $\rmC^*$-algebra is \emph{GNS-essential} if the image of the associated GNS representation
contains no non-zero compact operators.

\begin{theorem}\label{thm-rigid}
Suppose that
there exists a net $(b_i)_i$ in $A_1$ such that
$\limsup_i \varphi_1 (b_i b_i^*) \leq 1$,
$\lim_i \varphi_1 (b_i^*x b_i) = \varphi_1(x)$ and $\lim_i \varphi_1(xb_i) =0$ for $x \in A_1$
and  $\varphi_2$ is GNS-essential.
Then, the embedding $A \subset \DT (A, \varphi)$ is rigid, i.e.,
if $\Phi$ is a completely positive map on $\DT(A, \varphi)$ extending the identity map on $A$, then $\Phi = \id$.
\end{theorem}

The assumption on $\varphi_1$ in the theorem is satisfied when $\varphi_1$ is a GNS-essential tracial state by Glimm's lemma.
More generally, one can check the assumption on $\varphi_1$ when the centralizer of $\varphi_1$ is diffuse in a suitable sense (see Lemma \ref{lem-diffuse}). 
Combing the theorem with a modification of Glimm's lemma (Lemma \ref{lem-Ozawa}), told to us by Narutaka Ozawa,
we obtain 

\begin{theorem}\label{thm-nomimi}
If
$\varphi_1$ is faithful,
$\varphi_2$ is GNS-essential,
and $A_1$ contains no non-zero projection $p$ such that $pAp = \lC p$,
then the embedding $A \subset \DT (A, \varphi)$ is rigid.
\end{theorem}

By a result in \cite{Hasegawa-BS}, $\DT(A, \varphi)$ is nuclear if and only if so are both $A_1$ and $A_2$.
Thus, this theorem provides new examples of rigid embedding of exact $\rmC^*$-algebras into nuclear ones.
One of key ingredients of our proof is an approximation result (Proposition \ref{prop-freeglimm}) of certain states on $\DT(A, \varphi)$.
This proposition is inspired from a geometric property of compactifications of locally infinite trees
and the proof is based on Glimm's lemma (see the remark after Lemma \ref{lem-Psicpt}).
Using the proposition we also show that $\DT(A, \varphi)$ is simple and purely infinite whenever both $\varphi_1$ and $\varphi_2$ are GNS-essential.
This corresponds to Laca and Spielberg's result for pure infiniteness of crossed products associated to strong boundary actions \cite{Laca-Spielberg}.

\section{preliminaries}\label{section-pre}
Let $(A,E) =(A_1,E_1)\star_D(A_2,E_2)$ be a reduced amalgamated free product with non-degenerate conditional expectations $E_1$ and $E_2$ (see \cite{Voiculescu}).
We may identify the index set $ \{1, 2\}$ with $\lZ/2\lZ$.
We always assume that $A_1 \neq D \neq A_2$.
For any $a \in A_i$, we set $a^\circ :=a - E_i(a)$, which belongs to $A_i^\circ := \ker E_i$.
Then, $\DT(A,E)$ is a $\rmC^*$-algebra generated by $A$ and projections $e_1$ and $e_2$ such that
\[
e_1 +e_2 = 1, \quad \quad e_i a e_i = E_i (a)e_i \quad \text{for }\;i=1,2, \; a \in A_i .
\]
In \cite{Hasegawa-BS} it was shown that $\DT(A,E)$ is universal with respect to the above relations,
but we will not use this universal property in the present paper.

For later purposes, let us recall the construction of $\DT(A, E)$.
We call any element of the form $a_1 a_2 \cdots a_n$ for $a_j \in A_{i_j}^\circ$ with $i_{j} \neq i_{j+1}$, $j =1,\dots,n-1$
a \emph{reduced word of length $n$}.
Recall that the canonical conditional expectation $E_{A_i} \colon A \to A_i$ is given by $E_{A_i} (x) =0$ when $x$ is either in $A_{i+1}^\circ$ or a reduced word of length $n \geq 2$.
It follows that $E_i \circ E_{A_i} = E$.
Let $(L^2(A, E_{A_i}), \phi_i, \eta_i )$ be the GNS representation associated with $E_{A_i}$
and consider the $A$-$A_1 \oplus A_2$ $\rmC^*$-correspondence $(Y,\phi_Y) := (L^2(A, E_{A_1}) \oplus L^2(A, E_{A_2}), \phi_1 \oplus \phi_{2})$.
We denote by $\lL (Y)$ the $\rmC^*$-algebra of adjointable operators on $Y$.
Then, $\DT(A,E)$ is defined to be the $\rmC^*$-subalgebra of $\lL (Y)$ generated by $\phi_Y(A)$ and two projections $P_1$ and $P_2$,
where
$P_1$ is the projection onto the closure of
\[
( \eta_1 A_1  \oplus A_1^\circ A_2^\circ \eta_1 A_1 \oplus A_1^\circ A_2^\circ A_1^\circ A_2^\circ\eta_1 A_1  \oplus \cdots )
 \oplus ( A_1^\circ \eta_2 A_2 \oplus A_1^\circ A_2^\circ A_1^\circ \eta_2 A_2  \oplus \cdots )
\]
and the range of $P_2 = 1 - P_1$ is the closure of
\[
 ( A_2^\circ \eta_1 A_1 \oplus A_2^\circ A_1^\circ A_2^\circ \eta_1 A_1  \oplus \cdots ) 
 \oplus ( \eta_2 A_2   \oplus A_2^\circ A_1^\circ \eta_2 A_2   \oplus A_2^\circ A_1^\circ A_2^\circ A_1^\circ \eta_2 A_2    \oplus \cdots ).
\]
It is easy to check that $P_i^\perp \phi_Y(a) P_i^\perp =\phi_Y(E_i(a)) P_i^\perp$ for $a \in A_i$ and $i=1,2$.
We may omit $\phi_Y$ and set $e_i := 1 - P_i$.

Note that $a e_i = e_{i+1} a e_i$ holds for $a \in A_i^\circ$ and $i=1,2$.
We put $t(a):= a e_i$, and may write $t_i(a)$ when we emphasize that $a$ is in $A_i$.
We observe that $t_i(a)^*t_j (b) = \delta_{i,j} E_i(a^*b) e_{i}$ for $a \in A_i^\circ$, $b\in A_j^\circ$ and $i,j \in \{1,2\}$.
Similarly, for any reduced word $a_1a_2 \cdots a_n$ of length $n$ with $a_n \in A_i$,
we set $t (a_1 a_2 \cdots a_n) := a_1 a_2 \cdots a_n e_{i} = t (a_1) t (a_2) \cdots t (a_n)$.
For each $n,m \geq 1$, we denote by $\calF_{n,m}$ the linear span of all the elements of the form $t (x) z t(y^*)$,
where $x = x_1 x_2 \cdots x_n$ and $y=y_1y_2 \cdots y_m$ are reduced words of length $n$ and $m$, respectively
such that $x_n,y_m \in A_i^\circ$ for some $i$, and $z \in A_{i+1}$.
Similarly, we define $\calF_{n,0}$ to be the linear span of the set of elements of the form $t (x) z e_i$ with
$x =x_1 \cdots x_n$ is a reduced word of length $n$ such that $x_n \in A_i^\circ$ and $z\in A_{i+1}$
and set $\calF_{0,n} := \{ x^* \mid x \in \calF_{n,0}\}$.
Finally, we set $\calF_{0,0} = \{0 \}$ for convenience.
The next proposition is essentially proved in \cite{Hasegawa-BS}, but we give a sketch of the proof for the reader's convenience.

\begin{proposition}\label{thm-str}
The subspace
\[
e_1^\perp A_1 e_1^\perp + e_2^\perp A_2 e_2^\perp +  \lspan \bigcup_{n,m\geq 0} \calF_{n,m}.
\]
is a norm dense $*$-subalgebra of $\DT(A,E)$.
There exists a unique conditional expectation $E_{A_i}^\sim$ from $\DT(A,E)$ onto $A_i$ for $i=1,2$
which extends $E_{A_i} \colon A \to A_i$ and satisfies that
\[
E_{A_i}^\sim (e_{i}^\perp a e_{i}^\perp ) = a  \quad \text{for }\; a \in A_i, \quad \quad
E_{A_i}^\sim (e_{i} x e_{i} ) = 0 \quad \quad \text{for }\; x \in \DT(A, E) ,
\]
and $E_{A_i}^\sim =0$ on $\bigcup_{n,m\geq 0} \calF_{n,m}$.
The direct sum of the GNS representations associated with $E_{A_1}^\sim$ and $E_{A_2}^\sim$ is faithful.
\end{proposition}
\begin{proof}

The ucp (unital completely positive) map $E_{A_i}^\sim \colon \DT(A,E) \to A_i$ given by $x \mapsto \i< x \eta_i , \eta_i >$
defines a conditional expectation onto $A_i$ which obviously extends $E_{A_i}$.
Since $e_i^\perp \eta_i = P_i  \eta_i  = \eta_i $ and $e_i\eta_i = P_{i+1} \eta_i  =0$ hold,
it follows that $E_{A_i}^\sim (e_i^\perp a e_i^\perp) =E_{A_i} (a)$ and $E_{A_i}^\sim ( e_i x e_i) =0$ for $a \in A_i$ and $x \in \DT(A,E)$.
It is also not hard to see that $E_{A_i}^\sim$ vanishes on $\bigcup_{n,m} \calF_{n,m}$.
Note that
the inclusion map $\DT(A,E) \hookrightarrow \lL (Y)$ is  nothing but the direct sum of the GNS representations associated with $E_{A_1}^\sim$ and $E_{A_2}^\sim$.

The density of the subspace $\calA := e_1^\perp A_1 e_1^\perp + e_2^\perp A_2 e_2^\perp +  \lspan \bigcup_{n,m\geq 0} \calF_{n,m}$
easily follows from the Cuntz--Pimsner algebra structure of $\DT(A,E)$ (see \cite{Hasegawa-BS}),
but we give here a more direct and elementary proof.
Since $\DT(A,E)$ is the norm closure of the $*$-algebra generated by $A_1,A_2$ and $e_1$,
it suffices to show that $\calA$
is a $*$-algebra containing $A_1,A_2$ and $e_1$.
It is obvious that $e_1 \in \calA$.
For any $a \in A_1$, we have
\[
a = e_2 a e_2 + e_2 a e_1 + e_1 a e_2 + E_1 (a)e_1 = e_2 a e_2 + E_1 (a) e_1 + t (a^\circ ) + t (a^{\circ*} )^*,
\]
which belongs to $\calA$.
Similarly, one has $A_2 \subset \calA$.
To see that $\calA$ is a $*$-algebra,
let  $a \in A_i$ and $b \in A_j^\circ$ be arbitrary elements.
Then, one has $(e_i^\perp a e_i^\perp) t (b) =  \delta_{i,j} e_{i+1}ab e_i = \delta_{i,j}t ( (ab)^\circ) \in \calA$.
Also, $(e_i^\perp a e_i^\perp ) (e_j^\perp b e_j^\perp) = \delta_{i,j} (e_i^\perp ab e_i^\perp  + e_i^\perp a e_i b e_i^\perp) \in \calA$. 
This shows, with the notation $\calF =\lspan \bigcup_{n,m\geq 0} \calF_{n,m}$,
that $e_i^\perp A_i e_i^\perp  \calF \subset \calA$ and $\calF e_i^\perp A_i e_i^\perp \subset \calA $ for $i=1,2$.
Finally, the inclusion $\calF \calF \subset \calA$ follows from the fact that $t(a)^*t(b) = \delta_{i,j} E_i(a^*b) e_i$ for $a \in A_i^\circ$ and $b \in A_j^\circ$.
\end{proof}

\section{Approximation of states}

Let $(A,\varphi) = (A_1, \varphi_1) \star (A_2, \varphi)$ be any reduced free product with non-degenerate states $\varphi_1$ and $\varphi_2$.
Denote by $( \calH, \pi_\varphi, \xi_\varphi)$ the GNS representation associated with $\varphi$.
We will use the following two representations of $\DT(A, \varphi)$ on $\calH$.
For $i=1,2$,
let $\sigma_i \colon \DT(A,E) \to \lB ( \calH)$ be the representation induced from the `$\otimes 1$ map' $\lL (L^2(A, E_{A_i})) \to
\lB (L^2 (A, E_{A_i}) \otimes_{A_i} \calH_i ) \cong \lB (\calH)$,
where $\calH_i$ is the GNS Hilbert space for $\varphi_i$.
We identify $\calH_i$ with $\overline{A_i \xi_\varphi} \subset \calH$.
Note that $\sigma_i = \pi_\varphi$ on $A$.
Thus, to simplify the notation, we omit $\pi_\varphi$ and write $\sigma (a) =a$ for $a \in A$.
To see the range of the projections $\sigma_1(e_1)$ and $\sigma_2 (e_1)$,
let $e_\varphi$ and $P_{i\to}$ denote the projections onto $\lC \xi_\varphi$
and
\[
\calH (i \!\to ) :=\overline{A_i^\circ \xi_\varphi \oplus A_i^\circ A_{i+1}^\circ \xi_\varphi \oplus A_i^\circ A_{i+1}^\circ A_i^\circ \xi_\varphi \oplus \cdots},
\]
respectively.
Notice that $\calH = \lC \xi_\varphi \oplus \calH (1\!\to) \oplus \calH (2\!\to)$.
Then, it follows that
\begin{equation}\label{eq-sigmae}
\sigma_i (e_i^\perp ) = e_\varphi  +  P_{i\to}, \quad \; \sigma_{i+1} (e_i^\perp) = P_{i\to}
\end{equation}
for $i=1,2$.

\begin{lemma}\label{lem-Psicpt}
With the above notion, the following hold true:
\begin{itemize}
\item[(i)] Let $P_{A_i} \in \lB(\calH)$ be the projection onto $\calH_i$.
Then, it follows that
$P_{A_i} \sigma_i(x) P_{A_i} = E_{A_i}^\sim (x) P_{A_i}$ and $P_{A_i}\sigma_{i+1} (x) P_{A_i} - E_{A_i}^\sim (x) P_{A_i} \in \lK (\calH_i)$
for $x \in \DT (A,\varphi)$.
\item[(ii)] Let $\psi_i$ be the state on $\DT(A,E)$ defined by $\psi_i (x) = \i< \sigma_i (x) \xi_\varphi, \xi_\varphi>$ for $x \in \DT(A,\varphi)$.
Then, it follows that $\psi_i = \psi_i(e_i^\perp (\,\cdot\,) e_i^\perp ) = \varphi \circ E_{A_i}^\sim$ and $\psi_i |_A = \varphi$.
\end{itemize}
\end{lemma}
\begin{proof}
By symmetry, we may assume that $i=1$.
We prove (i):
The first assertion immediately follows from the definition of $E_{A_i}^\sim$ (see Proposition \ref{thm-str}).
Since $\sigma_1(e_1) - \sigma_2(e_1) =-e_\varphi$ is compact,
every element $x \in \DT(A, \varphi)$ satisfies that $\sigma_1 (x) - \sigma_2 (x) \in \lK(\calH)$.
In particular, its corner $E_{A_1}^\sim (x) P_{A_1} - P_{A_1} \sigma_2 (x) P_{A_1} $ belongs to $\lK (\calH_1)$.
We prove (ii):
Since $\sigma_1 (e_1^\perp ) = e_\varphi + P_{1\to} \geq P_{A_1}$,
we have $\xi_\varphi = \sigma_1 (e_1^\perp ) \xi_\varphi = P_{A_1}\xi_\varphi$, implying that $\psi_1 = \psi_1 (e_1^\perp (\,\cdot\, )e_1^\perp ) = \varphi \circ E_{A_1}^\sim$.
Since $\sigma_1 = \sigma_2 = \pi_\varphi$ holds on $A$, it follows that $\psi_1 = \psi_2 = \varphi$ on $A$.
\end{proof}

When the reduced free product $(A,\varphi)$ comes from the reduced group $\rmC^*$-algebra
of a free product group $\Gamma = \Gamma_1 \ast \Gamma_2$,
the $\rmC^*$-algebra $\DT(A,\varphi)$ is identified with $C(\DT) \rtimes_{\rm r} \Gamma_1 \ast \Gamma_2$,
where $\DT$ is the compactification of the Bass--Serre tree associated with $\Gamma_1 \ast \Gamma_2$ (see \cite{Hasegawa-BS}).
In this case, the state $\psi_i$ is the composition of the canonical conditional expectation $C(\DT) \rtimes_{\rm r} \Gamma_1 \ast \Gamma_2 \to C(\DT)$
and the evaluation map $\delta_{\Gamma_i} \colon  C(\DT)  \to \lC; f \mapsto f (\Gamma_i)$.
Here $\Gamma_i$ is viewed as an element in $\bfT = \Gamma/\Gamma_1 \sqcup \Gamma/\Gamma_2$.
When $\Gamma_1$ is infinite and $(g_n)_n$ is any sequence of mutually distinct elements in $\Gamma_1$,
it follows from the definition of the topology of $\DT$ that $\lim_{n \to \infty} g_n \Gamma_2 = \Gamma_1$.
This shows that $\psi_2 (g_n^{-1} \cdot  g_n )$ converges to $\psi_1$ in the weak$^*$-topology. 
We will prove an analogues result below for reduced free products with respect to GNS-essential states
based on Glimm's lemma.
The following easily follows from Glimm's lemma (see, e.g. \cite[Theorem 1.4.11]{Brown-Ozawa}).
\begin{lemma}[Glimm's lemma]\label{lem-Glimm}
Let $\phi$ be a GNS-essential state on a unital $\rmC^*$-algebra $B$.
Then, there exists a net $(a_i)_i$ in $\ker \phi$ such that
$\phi (a_i^*a_i) =1$ for $i$, $\lim_i \phi (x a_i) =0$
and $\lim_{i} \phi (a_i^*xa_i) = \phi (x)$ for any $x \in B$.
\end{lemma}

\begin{proposition}\label{prop-freeglimm}
Assume that $\varphi_1$ is GNS-essential and let $(a_i)_i$ be a net in $A_1^\circ$ as in Lemma \ref{lem-Glimm}.
Then, for any $x \in \DT(A,\varphi)$ it follows that $\lim_{i}\| e_1 a_i^* x a_i e_1 - \psi_1 (x) e_1 \| =0$.
In particular, $\lim_i  \psi_1 (b^*a_i^* x a_ic )= \psi_1 (x) \varphi (b^*c )$ holds for any $b,c \in A_2^\circ$.
\end{proposition}
\begin{proof}
Set $\theta_i (x) = e_1 a_i^* x a_i e_1 = t_1 (a_i)^* x t _1 (a_i)$ for $x \in \DT(A,\varphi )$.
Since $\theta_i$ has the norm $\|\theta_i (1)\| =\varphi (a_i^*a_i) =1$,
it suffices to show the desired convergence on the dense subset $e_1^\perp A_1e_1^\perp + e_2^\perp A_2 e_2^\perp  + \bigcup_{n,m \geq 1}\calF_{n,m}$
thanks to Proposition \ref{thm-str}.
For any $x \in A_1$, Lemma \ref{lem-Psicpt} (ii) shows that
$\psi_1 (e_1^\perp x e_1^\perp ) = \psi_1 (x) =  \varphi_1 (x)$.
Since $e_1 a_i^* e_2 = e_1 a_i^*$ holds, 
we have
\[
\theta_i (e_1^\perp x e_1^\perp ) = e_1 a_i^* e_2 x e_2 a_i e_1 = e_1 a_i^*x a_i e_1 = \varphi_1 (a_i^* x a_i ) e_1,
\]
which converges to $\varphi_1 (x) e_1$.
By Proposition \ref{thm-str} again,
we observe that $\psi_1 = \varphi \circ E_{A_1}^\sim$ vanishes on $e_2^\perp A_2 e_2^\perp$ and $\calF_{n,m}$ for $n,m \geq 0$.
On the other hand, $\theta_i$'s also vanish on $e_2^\perp \DT(A,\varphi )e_2^\perp$ since $t_1 (a_1) = e_2t_1 (a_1)$.
Thus, to see the first assertion, it is enough to show that $\lim_i \|\theta_i (x)\| =0$ for any $x \in \calF_{n,m}$, $n,m \geq 0$.
This follows from the observation that
for any $x \in A_j^\circ$, the norms $\| t_1^* (a_i) t_j (x) \| =\delta_{1,j} |\varphi_1 (a_i^* x) |$
and $\|t_j(x)^* t_1 (a_i) \| =\delta_{1,j} | \varphi_1 (x^* a_i) |$.

To see the second assertion, let $b,c \in A_2^\circ$ be arbitrary elements.
It follows from Lemma \ref{lem-Psicpt} (ii) that
$\psi_1 (b^*a_i^* x a_i c ) = \psi_1 (e_2 b^* a_i^* x a_i c e_2 ) = \psi_1 (e_2 b^* e_1 a_i^* x a_i e_1 c e_2) \longrightarrow \varphi (b^* c) \psi_1 (x)$.
\end{proof}

\begin{lemma}\label{lem-ess-ffl}
If $\varphi_1$ is GNS-essential,
then $\sigma_2$ is faithful.
\end{lemma}
\begin{proof}
Suppose that $\sigma_2$ is not faithful
and take a positive element $x \in \ker \sigma_2 $ of norm-one.
By Lemma \ref{lem-Psicpt} (i) $\sigma_1(x)$ is in $\lK (\calH)$.
Since $\sigma_1 \oplus \sigma_2$ is faithful, there exists $y \in A$ such that $\psi_1 (y^*x y ) =1$.
Let $(a_i)_i$ be as in Lemma \ref{lem-Glimm} for $\varphi_1$
and $b \in A_2^\circ$ be such that $\varphi (b^*b) =1$.
Then, Proposition \ref{prop-freeglimm} implies that $\lim_{i} \psi_1 (b^*a_i^* y^*xy a_i b ) = \psi_1 (y^*x y) =1$.
However, since $\sigma_1 (y^*xy)$ is also in $\lK (\calH)$ and $(a_i b \xi_\varphi )_i$ converges to 0 weakly,
we have $|\psi_1 (b^*a_i^* y^*xy a_i b ) | \leq \|\sigma_1 (y^*x y) a_i b \xi_\varphi \| \longrightarrow 0$, a contradiction.
\end{proof}

\begin{theorem}
If $\varphi_1$ and $\varphi_2$ are GNS-essential,
then $\DT(A,\varphi)$ is purely infinite and simple.
\end{theorem}
\begin{proof}
Fix a positive element $x$ of norm-one in $\DT(A,\varphi)$.
We will show that there exists $w \in \DT(A,\varphi)$ such that $\| w^*xw - 1 \| < 1$.
By the previous lemma, $\sigma_1$ and $\sigma_2$ are both faithful.
Take $y_1,y_2 \in A$ so that $\psi_1 (y_1^* x y_1 ) = \psi_2 (y_2^*x y_2 ) = 1$.
Let $(a_i)_i$ and $(b_j)_j$ be nets in $A_1^\circ$ and $A_2^\circ$ obtained by Lemma \ref{lem-Glimm} for $\varphi_1$ and $\varphi_2$, respectively.
Then, by Proposition \ref{prop-freeglimm} there exist $i,j$ such that $\| e_1 a_i^* y_1^* x y_1 a_i e_1 - e_1 \| + \| e_2 b_j^* y_2^* x y_2 b_j e_2 - e_2 \| <1/2$.
We observe that $e_2 y_1^* x y_2 e_1$ and $e_1 y_2^* x y_1 e_1$ are in $\ospan \bigcup_{n,m \geq 0} \calF_{n,m}$ (see Proposition \ref{thm-str}).
By the same argument in the proof of Proposition \ref{prop-freeglimm},
one can check that $\lim_{i,j} \|t_1 (a_i)^* z  t_2 (b_j) \| =0$ for $z \in \ospan \bigcup_{n,m\geq 1} \calF_{n,m}$.
Thus, we may assume that $ \| t_1 (a_i)^* y_1^*x y_2  t_2 (b_j) \| + \| t_2 (b_j)^* y_2^*x y_1 t_1 (a_i) \| <1/2$.
Letting $w:= y_1 a_i e_1 + y_2 b_j e_2$ we have
\begin{align*}
\|  w^* x w -1 \|  & \leq \| e_1 a_i^* y_1^* x y_1 a_i e_1 - e_1 \| + \| e_2 b_j^* y_2^* x y_2 b_j e_2 - e_2 \|  \\
&\quad \quad \quad
+ \| t_1 (a_i)^* y_1^*x y_2  t_2 (b_j) \| + \| t_2 (b_j)^* y_2^*x y_1 t_1 (a_i) \| <1.
\end{align*}
This completes the proof.
\end{proof}

\section{Proof of Theorem}

In what follows, $(A, \varphi) = (A_1, \varphi_1) \star (A_2, \varphi_2)$ is
a reduced free product of unital $\rmC^*$-algebras.

\begin{lemma}
Let $\Phi$ be any ucp map on $\DT(A,\varphi)$ which is identical on $A$.
Then, for the positive contraction $x= \Phi (e_1)$, the following two inequalities hold:
\begin{equation}\label{eq-axa}
ax a^* \leq \varphi (a^*a) (1-x) \quad \text{for} \quad a \in A_1,\quad\quad
b(1-x) b^* \leq \varphi (b^*b) x \quad \text{for} \quad b \in A_2.
\end{equation}
\end{lemma}
\begin{proof}
Since $A$ sits in the multiplicative domain of $\Phi$ (see around \cite[Definition 1.5.8]{Brown-Ozawa}),
we have $\Phi (axb) = a \Phi(x) b$ for any $a,b \in A$.
Thus, for any $a \in A_1^\circ$ and $b \in A_2^\circ$
one has
$
ax a^* = \Phi (a e_1 a^*) = \Phi (e_2 ae_1 a^* e_2) \leq \| a e_1 a^* \| \Phi (1-e_1) = \varphi (a^*a ) (1-x)
$
and also
$
b (1-x) b^* = \Phi (b  e_2 a^*) = \Phi (e_1 be_2 b^* e_1) \leq \| b e_2 b^* \| \Phi (e_1) = \varphi (b^*b )x.
$
\end{proof}

For $m\geq 1$, we denote by  $P_{1,2,m}$ and $P_{2,2,m}$ the projections onto the closures of
\[
\overbrace{ (A_1^\circ A_2^\circ ) (A_1^\circ A_2^\circ )  \cdots (A_1^\circ A_2^\circ ) }^{m} \xi_\varphi
\quad \text{and} \quad 
\overbrace{ (A_2^\circ A_1^\circ ) (A_2^\circ A_1^\circ )  \cdots (A_2^\circ A_1^\circ ) }^{m}A_2^\circ  \xi_\varphi,
\]
respectively.
We define $P_{2,1,m}$ and $P_{1,1,m}$ in a similar way.
Then,
these projections are mutually orthogonal and one has
\[
P_{1 \to } = P_{A_1^\circ} + \bigoplus_{m \geq 1} P_{1,1,m} + P_{1,2,m},
\quad
P_{2\to } =P_{A_2^\circ} +  \bigoplus_{m \geq 1} P_{2,2,m} + P_{2,1,m}.
\]
Here we set $P_{A_i^\circ} := P_{A_i} - e_\varphi \colon \calH \to \calH_i^\circ := \overline{A_i^\circ \xi_\varphi}$.

%

\begin{lemma}\label{lem-corner}
Let $x \in \DT(A,\varphi)$ be a positive contraction satisfying
$E_{A_1}^\sim(x)=0$ and
(\ref{eq-axa}).
Then, for any $m \geq 1$ and $i =1,2$ it follows that
\begin{align*}
P_{A_1} \sigma_1 (x) P_{A_1} &=0 ,       &P_{A_2^\circ} \sigma_1 (x) P_{A_2^\circ} &= P_{A_2^\circ},\\
P_{1,i,m} \sigma_1 (x) P_{1,i,m} &= 0, &P_{2,i,m} \sigma_1(x) P_{2,i,m} &= P_{2,i,m}.
\end{align*}
Consequently, it follows that $\sigma_1 (x) =\sigma_1 (e_1)$.
\end{lemma}
\begin{proof}
We first note that $\psi_1 (x) = \varphi (E_{A_1}^\sim (x)) =0$ by Lemma \ref{lem-Psicpt} (ii).
By assumption we have $P_{A_1} \sigma_1 (x) P_{A_1} =E_{A_1}^\sim (x) = 0$.
We prove $P_{A_2^\circ} \sigma_1 (1 -x) P_{A_2^\circ} =0$.
For any $b \in A_2^\circ$ we have $\psi_1 (b^* (1-x) b) \leq \varphi (bb^*) \psi_1 (x)=0$.
The polarization identity shows that $P_{A_2^\circ} \sigma_1 (1-x) P_{A_2^\circ} =0$.

Fix $m \geq 1$ and an element $z =a_1 b_1 a_2 b_2 \cdots a_m b_m \in A_1^\circ A_2^\circ \cdots A_1^\circ A_2^\circ$.
Set  $w = b_1 a_2 b_2 \cdots a_m b_m$ so that $z = a_1 w$.
Then, applying (\ref{eq-axa}) 2$m$ times we obtain
\begin{align*}
0 \leq \psi_1 (z^* x z)
&\leq \varphi (a_1a_1^*) \psi_1 (w^* ( 1- x)w^* ) \\
&\leq  \varphi (a_1a_1^*) \varphi (b_1b_1^*) \psi_1 (b_m^* a_m^* \cdots a_2^* x a_2\cdots a_m b_m) \\
&\leq  \prod_{k=1}^m \varphi (a_ka_k^*) \varphi (b_kb_k^*)  \psi_1(x) =0,
\end{align*}
which implies $\psi_1 (z^* x z) = 0$ and $\varphi (a_1a_1^*) \psi_1 (w^* ( 1- x)w^* ) =0$.
By the polarization trick again, we get $P_{1,2,m} \sigma_1 (x) P_{1,2,m} = 0$ and $P_{2,2,m-1} \sigma_1 (1-x) P_{2,2,m-1} = 0$.
For any $a \in A_1^\circ$,
one has $z a \xi_\varphi \in P_{1,1,m} \calH$ and $wa \xi_\varphi \in P_{2,1,m}\calH$,
and applying (\ref{eq-axa}) again, we obtain
\begin{align*}
0 \leq \psi_1 (az^* x za)
\leq \varphi (a_1a_1^*) \psi_1 (a^*w ( 1- x)wa  )
\leq  \prod_{k=1}^{m} \varphi (a_ka_k^*) \varphi (b_kb_k^*)  \psi_1(a^*x a) =0.
\end{align*}
Here, we used $ \psi_1(a^*x a) = \varphi (E_{A_1}^\sim (a^*x a) ) =\varphi (a^*E_{A_1}^\sim(x)a )=0$.
Hence, we have $P_{1,1,m} \sigma_1 (x) P_{1,1,m} = 0$ and $P_{2,1,m} \sigma_1(1-x) P_{2,1,m} = 0$.

To see the second assertion,
recall that $\sigma_1 (e_1) = P_{A_2^\circ} + \sum_{m \geq 1} P_{2,2,m} +P_{2,1,m}$ (see Eq.~(\ref{eq-sigmae})).
We observe that $\sigma_1(x) (P_{A_1} + \sum_{m \geq 1} P_{1,1,m} +P_{1,2,m})=0$ since $\|x p \|^2 = \| p x^2 p \| \leq \|pxp \|$ holds for any projection $p$.
Thus, $\sigma_1 (x) = \sigma_1 (e_1 x e_1) \leq \sigma_1 (e_1)$.
For any $q \in \{ P_{A_2^\circ}\} \cup \{ P_{2,i,m}\}_{m \in \lN, i=1,2}$ we also have $\| \sigma_1 (e_1 - e_1 x e_1) q \|^2  = \| q  \sigma_1 (e_1 - e_1 x e_1)^2 q \| \leq \| q \sigma_1 (e_1 - e_1 x e_1 )q \| =0$.
This shows $\sigma_1 (x) = \sigma_1 (e_1)$.
\end{proof}

\begin{lemma}\label{lem-bddbelow}
Let $B$ be an infinite dimensional unital $\rmC^*$-algebra and $\phi$ be a state on $\phi$ with faithful GNS representation. Then, there is no constant $R >0$ such that $ \|a^*a\| \leq R \phi(a^*a)$ holds for all $a \in \ker \phi $.
\end{lemma}
\begin{proof}
Let $(\calH_\phi, \pi_\phi, \xi_\phi)$ be the GNS representation associated with $\phi$. On the contrary, suppose that there exists a constant $R>0$ such that $\| a \| \leq R\| a \xi_\phi \|$ holds for $a \in \ker \phi $. Then, for any $a \in B$, we have
$\|a  \| \leq (R +1) (\|a^\circ \xi_\phi \| + | \phi(a) | ) \leq 2(R+1) \| a \xi_\phi \|$.
Since $B$ is infinite dimensional, the canonical inclusion $B \subset B^{**}$ is proper. Hence, one can find an element $x \in B^{**} \setminus B$ and a bounded net $a_\lambda \in B$ such that $a_\lambda$ converges to $x$ strongly by the Kaplansky density theorem.
Then, the net $(a_\lambda \xi_\phi)_\lambda$ is a Cauchy net.
The above estimate implies that $(a_\lambda)_\lambda$ is also a Cauchy net,
and thus $x = \lim_\lambda a_\lambda \in B$, a contradiction.
\end{proof}

We are now ready to prove the main theorem.
\begin{proof}[Proof of Theorem \ref{thm-rigid}]
Let $\Phi$ be any ucp map on $\DT(A,\varphi)$ such that $\Phi |_A = \id$.
Since $\DT(A,\varphi)$ is generated by $A$ and $e_1$, it is enough to show that $\Phi (e_1 ) =e_1$.
The positive contraction $x := \Phi (e_1)$ satisfies (\ref{eq-axa}).
Since $\sigma_1$ is faithful by Lemma \ref{lem-ess-ffl}, 
it suffices to show that $E_{A_1}^\sim (x) =0$ thanks to Lemma \ref{lem-corner}.

We first show that $E_{A_2}^\sim (x) = \psi_1(1- x)1$.
Let $(b_i)_i$ be a net in $A_1$ as in Theorem \ref{thm-rigid}.
We may assume that $\varphi_1(b_i)=0$ and $\varphi_1(b_i^*b_i) =1$ for all $i$.
Then,
for any $a \in A_2^\circ$,
applying Proposition \ref{prop-freeglimm} to $(b_i)_i$ and using (\ref{eq-axa}) twice, we have
\[
\psi_1(x)\varphi_2(a^*a) = \lim_i \psi_1(a^*b_i^* x b_i a ) \leq \limsup_i \varphi_1(b_ib_i^*) \psi_1 (a^* (1-x) a) \leq \psi_1 (x)\varphi_2 (aa^*) .
\]
Replacing $a$ by $a^*$ 
we obtain $\psi_1(x) \varphi_2(a^*a) = \psi_1 (x) \varphi_2 (aa^*) =\psi_1(a^*(1-x)a ) $ for all $a \in A_2^\circ$.
By the polarization identify, this implies that $(P_{A_2}-e_\varphi )\sigma_1(1-x) (P_{A_2}-e_\varphi) =\psi_1(x)$ on $\calH_2^\circ$.
It follows from Lemma \ref{lem-Psicpt} (i) and $e_\varphi$ being compact
that $E_{A_2}^\sim (1-x) - \psi_1(x)1 |_{\calH_2 } = A_2 \cap \lK(\calH_2 )$.
Since $\varphi_2$ is GNS-essential,
we have $E_{A_2}^\sim (x) = \psi_1 (1-x) 1$.
We also obtain $\psi_2(x) =\psi_2 (E_{A_2}^\sim (x)) = \psi_1(1-x)$.
To see that $\psi_1(x) =0$,
take $a \in A_2^\circ $ arbitrarily.
We then have
\begin{align*}
\psi_1 (x) aa^*
= a E_{A_2}^\sim (1-x) a^*
=  E_{A_2}^\sim (a (1-x) a^*) 
\leq \varphi (a^*a) E_{A_2}^\sim ( x )
=\psi_1(1-x) \varphi (a^*a).
\end{align*}
Since $A_2$ is infinite dimensional, Lemma \ref{lem-bddbelow} implies that $\psi_1 (x) = 0$ and $\psi_2 (x) =1$.

Now, for any $a \in A_1^\circ$,
we have $\psi_2(a^* x a ) \leq \varphi_1(aa^*) \psi_2 ( 1- x) =  0$.
By the polarization trick and Lemma \ref{lem-Psicpt} again,
it follows that $E_{A_1}^\sim (x)
\in A_1 \cap \lK (\calH_1 )$.
Since $\varphi_1$ is GNS-essential, we conclude that $E_{A_1}^\sim (x) =0$.
\end{proof}

Theorem \ref{thm-nomimi} follows from the following lemma, which was told us by Narutaka Ozawa.
\begin{lemma}\label{lem-Ozawa}
Let $\phi$ be a faithful  state on a $\rmC^*$-algebra $B$ which contains no non-zero projection $p$ such that $pBp =\lC p$.
Then, there exists a net of positive elements $(b_i)_i$ in $B$ such that $\lim_i \phi (a b_i)=0$ and $\lim_i \phi (b_i^*a b_i) = \phi (a)$ for $a \in B$.
\end{lemma}
\begin{proof}
We may assume that $B$ is unital and separable.
Let $F \subset A$ be a finite subset of norm-one elements and $\varepsilon>0$ be arbitrary.
By the Krein--Milman theorem, there exists a convex combination $\phi' = \sum_{k=1}^n \lambda_k \psi_k$ of pure states such that
$| \phi (a) - \phi' (a) | < \varepsilon$ for $a\in F$.
By the Akemann--Anderson--Pedersen excision theorem \cite{AAP},
we find a decreasing sequence of positive elements $(e_{k,m})_m$ of norm 1 for $k=1,\dots,n$
such that for sufficiently large $m$, $e_k:=e_{k,m}$ satisfies $\| \psi_k (a)e_{k}^2 - e_k a e_k \|< \varepsilon$ for $a \in F$.
Let $(\calH_\phi, \pi_\phi, \xi_\phi)$ be the GNS representation associated with $\phi$.
Set $b_0 :=1$.
We will find positive elements $b_k$, recursively for $k=1,\dots, n$ such that $\phi (b_k^2) =1$, $\| e_k b_k \xi_\phi - b_k \xi_\phi \| < \varepsilon/{n^2}$
and $|\phi (b_j a b_k) | < \varepsilon$ for $a \in F$ and $j=0,1\dots, k-1$.
To see this, one observes that $e_{k,m}$ can be taken as $1- h_k^{1/m}$, where $h_k$ is a strictly positive element in the hereditary subalgebra $\ker \psi_k \cap (\ker \psi_k)^*$.
Then, $\pi_\phi(1 - h_k^{1/m})$ converges to the projection $p$ onto $\ker \pi_\phi(h)$ strongly.
Observe that $p \pi_\phi (a) p = \psi_k (a)p$ for $a \in B$.
Thus, the assumption implies either $p=0$ or $p \notin \pi_\phi(B)$.
Since $h_k$ is not invertible in $B$, 
$0$ is not isolated in the spectrum $\spec (h_k)$ of $h_k$.
Take a decreasing sequence $(\delta_j)_j$ in $\spec (h) \setminus \{0 \}$ such that $\lim_{j\to \infty} \delta_j =0$.
For each $j$, choose a positive function $f_j \in C(\spec (h)) \cong \rmC^*(h,1)$ such that
$f_j(\delta_j) =1$ and $f_j $ vanishes outside $( (\delta_{j+1}+\delta_j )/2, (\delta_j +\delta_{j-1})/2 )$.
Since $\phi$ is faithful, $\phi (f_j^2)$ is non-zero.
By the choice of $f_k$'s, one concludes that $( \phi(f_j^2)^{-1/2} f_j \xi_\phi )_ j$ is an orthonormal system.
We also have $\| h_k^{1/m} \phi(f_j^2)^{-1/2} f_j \xi_\phi  \| \leq \delta_{j-1}^{1/m} \longrightarrow 0$ as $j \to \infty$.
Thus, for sufficiently large $j$, $b_k:=\phi(f_j^2)^{-1/2}f_j$ satisfies
$\| e_k b_k \xi_\phi - b_k \xi_\phi \| < \varepsilon$
and $|\phi (b_j a b_k) | < \varepsilon/{n^2}$ for $a \in F$ and $j=0,\dots, k-1$.
Now, letting $b:= \sum_{k=1}^n \lambda_k^{1/2} b_k$ we have
\[
\phi(b a b )  \approx_\varepsilon \sum_{k=1}^n \lambda_k \phi (b_k a b_k) \approx_{2\varepsilon} \sum_{k=1}^n \lambda_k \phi (b_k e_k a e_k b_k)
\approx_\varepsilon \sum_{k=1}^n \lambda_k  \psi_k(a) \phi (b_k e_k^2  b_k) \approx_{2\varepsilon} \phi' (a) \approx_\varepsilon \phi (a).
\]
and $|\phi (a b) | < \varepsilon$ for $a \in F$.
\end{proof}

Finally, we close the paper by giving another sufficient condition for the assumption of Theorem \ref{thm-rigid}
which is applicable for possibly non-faithful states.
\begin{lemma}\label{lem-diffuse}
Let $\phi$ be a state on a unital $\rmC^*$-algebra $B$ and $(\calH_\phi, \pi_\phi, \xi_\phi)$ be its GNS representation.
Suppose either
\begin{itemize}
\item[(i)]there exists a unital $\rmC^*$-subalgebra $B_0$ of the centralizer $\{x \in B \mid \phi(x b) = \phi(bx) \text{ for } b \in B \}$
of $\phi$ with $\phi$-preserving conditional expectation $E \colon B \to B_0$
such that $\phi|_{B_0}$ is GNS-essential, or
\item[(ii)]there exists a unitary $u$ in the von Neumann algebra $\pi_{\phi}(B)''$ such that $\i<u^k \xi_{\phi},\xi_{\phi}>=0$ for $k\in \lZ \setminus \{0\}$ and $\i<u^*xu \xi_{\phi}, \xi_{\phi}> = \phi(a)$ for $x \in \pi_{\phi}(B)''$.
\end{itemize}
Then, there exists a net $(b_i)_i$ in $B$ such that $\lim_i \phi (b_i b_i^*) =1$, $\lim_i \phi (b_i^*x b_i ) =\phi(x)$
and $\lim_i \phi(xb_i) =0$
for $x \in  B$.
\end{lemma}
\begin{proof}
We may assume that $\pi_\phi$ is faithful and $B \subset \lB (\calH_\phi)$.
Assume that (i) holds.
Take a net $(b_i)_i$ in $B_0$ as in Lemma \ref{lem-Glimm}.
Then, one has $\phi(b_i^*b_i) = \phi (b_ib_i^*) =1$ for all $i$
and
$\phi (b_i^*x b_i) = \phi (E (b_i^*xb_i)) = \phi (b_i^* E (x) b_i) \longrightarrow \phi(E (x) ) = \phi (x)$ for $x\in B$.

Assume that (ii) holds.
It is sufficient to show that
for any finite subset $F$ of the unit ball of $B$ and $\varepsilon>0$ arbitrarily,
there exists $b \in B$ such that $\phi(b^*b) = \phi (bb^*)$,  	
$|\phi (x b)| < \varepsilon$ and $| \phi (b^* x b ) - \phi (x) | < \varepsilon$ for $x \in F$.
By assumption, there exists a unitary $u$ in the centralizer $(B'')_\phi$ such that $\phi (u^n) =0$ for $n \neq 0$.
For any $\delta>0$ and $k \in \lN$,
by the Kaplansky density theorem,
there exists a unitary $v \in B$ such that $\|v \xi_\phi - u^k \xi_\phi \| < \delta$.
Then, it is easy to check that $b:= v$ is the desired one for sufficiently small $\delta$ and sufficiently large $k$.
\end{proof}
\begin{remark}
In the previous lemma, (i) does not imply (ii) in general.
For example, let $B =C([0,1])$ and $\phi$ be the state defined by the Lebesgue measure.
Then, $(\phi +\delta_0)/2$ is GNS-essential since $B$ has no non-trivial projections,
but the von Neumann algebra obtained by the GNS representation is isomorphic to $L^\infty ([0,1]) \oplus \lC$.
\end{remark}

\section*{Acknowledgment}
The author appreciates his supervisor, Yoshimichi Ueda, for his constant encouragement.
He is grateful to Narutaka Ozawa for showing Lemma \ref{lem-Ozawa} to him.
This work was supported by the Research Fellow of the Japan Society for the Promotion of Science.

\end{document}